%% file: hyperbolic_systems.tex
\documentclass[a4paper]{article}
\usepackage[english]{babel}
\usepackage{amssymb}
\usepackage[sumlimits,intlimits]{amsmath}
\usepackage{mathrsfs}
\usepackage{amsthm}
\usepackage{a4wide}

\input{cs_makros}

\begin{document}

\title{Symmetric hyperbolic systems in algebras of generalized functions
and distributional limits}

\author{G\"unther H\"ormann \& Christian Spreitzer\thanks{The author
acknowledges  the support of  FWF-project grants Y237 and P20525.}\\
\ \\
Fakult\"at f\"ur Mathematik\\
Universit\"at Wien\\
Austria}
\date{\today}

\maketitle

\begin{abstract} 
\noindent We study existence, uniqueness, and distributional
aspects of generalized solutions to the Cauchy problem for first-order
symmetric (or Hermitian) hyperbolic systems of partial differential
equations with Colombeau generalized functions as coefficients and data.
The proofs of solvability are based on refined energy estimates on
lens-shaped regions with spacelike boundaries. We obtain several variants
and also partial extensions of previous results in \cite{O:89,LO:91,GH:04}
and provide aspects accompanying related recent work in
\cite{O:08,GO:11,GO:11b}.



\end{abstract}

\section{Introduction}

In this paper we establish existence and uniqueness of a generalized
solution to the hyperbolic Cauchy problem
\begin{eqnarray}\label{ivp}
  \partial_t U +\sum_{j=1}^n A^j\partial_{x_j}U+BU &=& F
     \qquad \text{on } (0,T)\times \mathbb{R}^n,\\
     \label{ivp_ic}  U|_{t=0} &=& G,
\end{eqnarray}
where $U$, $F$ and $G$  are vectors of length $m$ and $A^j$ and $B$ are $m
\times m$-matrices whose components are generalized functions in the sense
of J.F. Colombeau (cf.\ \cite{Colombeau:84, Colombeau:85}). All 
coefficients and data may therefore represent functions or distributions
of low regularity. Our main focus as well as the essential methods are
following up along the lines of the seminal papers \cite{O:88,O:89,LO:91}.

Problems of the type (\ref{ivp})-(\ref{ivp_ic}) play a prominent role in 
models of wave propagation in highly heterogeneous media
with non-smooth variation of physical properties such as density, sound
speed etc. For more details on motivations from the natural
sciences and for further mathematical aspects in the context of the theory of
generalized functions we may refer to the papers mentioned above as well
as to the following series of papers on closely related research
\cite{CO:90,KH:01,HdH:01,GH:04,O:08,GO:11,GO:11b}. Second-order wave
equations in a similar mathematical context have been discussed in
\cite{Gramchev:94,VW:00,GMS:09}.

Besides existence and uniqueness of generalized solutions to the Cauchy
problem we are also interested in the relation of the unique Colombeau
solution to more classical and weak or distributional solution concepts,
if the coefficients are of compatible regularity. The analysis of such
questions and several convergence  results are also going back to earlier
investigations in \cite{O:88,O:89,LO:91,HdH:01}.

The outline of our paper is roughly as follows.
After a brief reminder of basic notions from Colombeau theory of
generalized functions in the following subsection, we devote a section to
the details of the construction of lens-shaped domains and on energy
estimates on such domains with explicit expressions for the constants.
These estimates are then the essential ingredients in proving several
variants of  existence and uniqueness results in Section 3. More precisely,
Theorem 3.1 is an extension of the main theorem in \cite[p. 98]{LO:91} to the case of complex matrices and relaxes the required constancy of the coefficients for large spatial distances to boundedness. Theorem 3.4 is based on $\mathcal{G}_{L^2}$-spaces and gives a result for systems of partial differential operators which is similar to \cite[Theorem 3]{GH:04} for the case of scalar pseudo-differential operators. 

The final section investigates regularity as well as compatibility of Colombeau-type solutions with classical and distributional solutions in case the
coefficient matrices are sufficiently regular. Proposition 4.1 is an analog of the compatibility proposition in \cite[p. 99]{LO:91} and \cite[Corollary 5]{GH:04}, whereas Proposition 4.2 is a $\mathcal{G}^\infty$-variant of the regularity result \cite[Proposition 6]{GH:04}. In Proposition 4.4 we establish convergence of the generalized solution to a weak solution for arbitrary Lipschitz continuous coefficients, thereby accompanying the case study with discontinuous coefficients in the acoustic transmission problem carried out in \cite[Theorem 2.4 and Corollary 2.5]{O:89}.

\emph{Basic notation and symbols:} Let $\Omega\subseteq \mathbb{R}^n$ be open, $1 \leq p \leq \infty$, $k \in \mathbb{N}_0$, then $W^{k,p}(\Omega)$ denotes the $L^p$-norm based Sobolev space of order $k$ on $\Omega$ and  $H^k(\Omega) = W^{k,2}(\Omega)$. For $s \in \mathbb{R}$, the Sobolev space $H^s(\mathbb{R}^n)$ is defined by Fourier transform. 
If $Y$ is a Banach space, then $C^k([0,T],Y)^m$ denotes the $m$-tuples of $k$ times continuously differentiable functions from the interval $[0,T]$ to $Y$. Similarly, $L^2([0,T],Y)^m$ denotes the $m$-tuples of square-integrable functions $[0,T] \to Y$ (in the Bochner-Lebesgue sense).  
If $R$ is a commutative ring with unit, then $M_m(R)$ denotes the ring of square matrices of size $m$ over $R$ with unit given by the identity matrix $\mathbb{I}_m$. For any $A \in M_m(\mathbb{C})$, the expression $\norm{A}{\mathrm{op}}$ denotes the operator norm of $A$ as linear map acting on $\mathbb{C}^m$. We use $\langle\:\cdot\,,\:\cdot\,\rangle$ to denote the standard scalar product on $\mathbb C^m$ and $\parallel\cdot\parallel$ for the Euclidean norm.

\subsection{Colombeau algebras of generalized functions}

This section serves to gather some basic notions from Colombeau theory of generalized functions. We adopt the topological viewpoint of the construction of generalized functions based on a locally convex vector space,  developed  in \cite{Garetto:05b}. For a comprehensive introduction to the theory of Colombeau algebras we refer to \cite{GKOS:01}.

Let $E$ be a locally convex topological vector space whose topology is given by the family of semi-norms $\{ p_j \}_{j\in J}$. The elements of 
$$ \mathcal M_E:=\{(u_{\e})_{\e}\in E^{(0,1]}:\forall j\in J \: \exists N \in \mathbb N_0\:\:\: p_j(u_{\e})=O(\e^{-N})\:\:\mm{as}\:\:\e\rightarrow 0 \}
$$
and
$$ \mathcal N_E:=\{(u_{\e})_{\e}\in E^{(0,1]}:\forall j\in J \: \forall q \in \mathbb N_0\:\:\: p_j(u_{\e})=O(\e^{q})\:\:\mm{as}\:\:\e\rightarrow 0 \}
$$
are called $E$-\emph{moderate} and $E$-\emph{negligible}, respectively. Defining operations componentwise turns $\mathcal N_E$ into a vector subspace of $\mathcal M_E$. We define the \emph{generalized functions based on $E$} as the quotient $\mathcal G_E:=\mathcal M_E / \mathcal N_E$. If $E$ is a differential algebra, then $\mathcal N_E$ is an ideal in $\mathcal M_E$ and $\mathcal G_E$ is a differential algebra as well, called the Colombeau algebra based on $E$.

Let $\Omega$ be an open subset of $\mathbb R^n$. By choosing $E= C^{\infty}(\Omega)$ with the topology of uniform convergence of all derivatives one obtains the so-called special Colombeau algebra $\mathcal G_{C^{\infty}(\Omega)}=\mathcal G(\Omega)$. In the current article  we will also use the space $E=H^{\infty}(\Omega)=\{h\in  C^{\infty}(\overline{\Omega}):\partial^{\alpha}h\in L^2(\Omega)\:\forall \alpha\in\mathbb N^{n}_0 \}$  with the family of semi-norms
$$ \norm{h}{H^k(\Omega)}=\big(\sum_{|\alpha|\leq k} \norm{\partial^{\alpha}h}{L^2(\Omega)}^2 \big)^{1/2}\quad(k\in \mathbb N_0),
$$
as well as $E=W^{\infty,\infty}(\Omega)=\{h\in C^{\infty}(\overline{\Omega}):\partial^{\alpha}h\in L^{\infty}(\Omega)\:\forall \alpha\in\mathbb N^{n} \}$
with the family of semi-norms
$$ \norm{h}{W^{k,\infty}(\Omega)}=\max\limits_{|\alpha|\leq k} \norm{\partial^{\alpha}h}{L^{\infty}(\Omega)}\quad(k\in \mathbb N_0),
$$
and
$E= C^{\infty}(\overline I\times\mathbb R^n)$, 
where $I$ is an open interval, equipped with semi-norms
$$\norm{h}{m,K}=\max\limits_{|\alpha|\leq m} \norm{\partial^{\alpha}h}{L^{\infty}( I \times K)}\quad(m\in \mathbb N_0, K\subset\subset \mathbb R^n). $$
 To avoid overloaded subscripts we use notations as in \cite{Hoermann:10} and denote
$$\mathcal G_{L^2}(\Omega):=\mathcal G_{H^{\infty}(\Omega)},\quad
\mathcal G_{L^{\infty}}(\Omega):=\mathcal G_{W^{\infty,\infty}(\Omega)}\quad\mm{and}\quad
\mathcal G(\overline I\times \mathbb R^n):=\mathcal G_{ C^{\infty}(\overline I\times\mathbb R^n)}.
$$

Colombeau algebras contain the distributions as a linear subspace. Their elements are equivalence classes of nets of smooth functions, $\mathcal G(\Omega)\ni u=[(u_{\e})_{\e}]$. We say that a Colombeau function $u$ is \emph{associated with a distribution} $w\in \mathcal D'(\Omega)$ if  some (and hence every) representative $(u_{\e})_{\e}$ converges to $w$ in $\mathcal D'(\Omega)$. The distribution $w$ represents the macroscopic behavior of $u$ and is called the \emph{distributional shadow} of $u$. Not every element of a Colombeau algebra is associated with a distribution.

In \cite{O:92},  the subalgebra $\mathcal G^{\infty}(\Omega)$ of \emph{regular generalized functions} in $\mathcal G(\Omega)$ was introduced to develop an intrinsic regularity theory in  $\mathcal G(\Omega)$. The subalgebra $\mathcal G^{\infty}_E$ of a Colombeau algebra $\mathcal G_{E}$ is obtained by demanding that the inverse $\e$-power $N$ in the moderateness estimates can be chosen uniformly over all derivatives (cf. Definition 25.1, Chapter VII in \cite{O:92}). For instance,  an element $u=[(u_{\e})_{\e}]\in \mathcal G(\Omega)$ belongs to $\mathcal G^{\infty}(\Omega)$ if and only if
$$\forall K\,\subset\subset \Omega\,\exists N\in\mathbb N_0\,\forall \alpha\in\mathbb N_0^n:\:\norm{\partial^{\alpha}u_{\e}}{L^{\infty}(K)}=O(\e^{-N}).$$
The subalgebra $\mathcal G^{\infty}(\Omega)$ plays the same role within $\mathcal G(\Omega)$ as $ C^{\infty}(\Omega)$ does within $\mathcal D^{\prime}(\Omega)$ and satisfies the important compatibility relation
$$ \mathcal G^{\infty}(\Omega)\cap \mathcal D^{\prime}(\Omega)= C^{\infty}(\Omega).
$$

\section{Standard lenses and basic energy estimates}

A central element of our proof of unique solvability of the Cauchy problem will be $L^2$-estimates performed on lens-shaped subsets of the strip $[0,T]\,\times \mathbb R^n$. Similar constructions have been used, e.g., in
\cite{BG:07}, Part I, Section 2.2,
in \cite{HE:73}, Section 4.3, 
and in \cite{Friedlander:75}, Section 4.4.
Since we are working in a  generalized functions setting, it is essential to have precise information on all dependencies of constants involved in these estimates. For this reason  we devote this section to the construction of a special variant of lens-shaped domains and to some basic estimates for these types of lenses. 

\begin{defi}\label{lensdef}
A standard lens $\mL$ of thickness $T>0$ and radii $0<R_1<R_2$ is the image set of the map $\psi:  [0,1]\times B_{R_2} \to \mathbb{R}^{n+1}$,
\begin{equation} 
\nn 
\psi(\Th,y)=
\begin{cases}
\quad(\Th T,y)\quad\quad\quad\quad\:\:\:\,\mm{for}\quad |y|\leq R_1\\
\quad(\Th T\frac{R_2-|y|}{R_2-R_1},y)\quad\quad\mm{for}\quad |y|>R_1
\end{cases}
\end{equation} 
where $B_{R_2}$ is the closed ball of
radius $R_2$, centered at the origin. We introduce slices of a lens, $\mHT:=\psi(\Theta,B_{R_2})$, as well as partial lenses
$\mL_{\Theta}=\bigcup_{0\leq\tau\leq\Theta}\mathcal H_{\tau}$ for $\Th\in(0,1]$.
The latter  are compact convex subsets of $[0,T]\times \mathbb R^{n}$ with Lipschitz continuous boundary $\partial\mLT=\mathcal H_0 \cup \mHT$. 
\end{defi}

The standard lens map $\psi$
introduced in Definition \ref{lensdef} as well as its restrictions $\psi_{\Theta}:B_{R_2}\rightarrow \mathcal H_{\Theta}$, $ y\mapsto\psi(\Th,y)$
are Lipschitz continuous, but not differentiable at points $(\Th,y)$ with $|y|=R_1$.  However, since the collection of these points is of Lebesgue measure zero with respect to $\psi_{\Th}(B_{R_2}) = \mathcal H_{\Theta}$,  they can be ignored when using the lens map to transform integrals over $\mL$ or $\mathcal H_{\Th}$ into integrals over the cylinder $[0,T]\times B_{R_2}$ or $B_{R_2}$ respectively (cf. \cite{Rudin:87}, Lemma 7.25 and 7.26). Smooth slices $\mathcal H_{\Th}$ are possible by an easy modification of the lens map, but not necessary for our considerations. 

\begin{lemma}\label{lens_int}
If $u\in C([0,T]\times \mathbb R^{n})$ and $\mathcal L$ is a standard lens of thickness $T$,
\begin{equation}
\nn
\int\limits_{\mL}|u|\,dV_{n+1} \leq T\!\!\sup_{\Theta\in[0,1]} \int\limits_{\mathcal H_{\Theta}} |u(\Th,\cdot)|\,dV_{n},\quad \! 0\leq\frac{d}{d\Theta}\int\limits_{\mLT}|u|\,dV_{n+1} \leq T\!\!\int\limits_{\mHT}|u(\Th,\cdot)|\,dV_{n}.
\end{equation}
\end{lemma}

\begin{proof}
First note that the map $\psi_{\Th}:B_{R_2}\rightarrow \mHT$ is a (global) parametrization of the slice $\mHT$. The volume density on $\mHT$ is 
$
\rho_{\Th}(y)
=
\sqrt{\mathrm{det}(D\psi_{\Th}(y)^TD\psi_{\Th}(y))}
=
\sqrt{\mathrm{det}(D_y\psi(\Th,y)^TD_y\psi(\Th,y))},
$
where
$D_y\psi$ is obtained from the Jacobian $D\psi$ by removing the first column, more precisely $D\psi=(D_{\Th}\psi\:D_y\psi)$. Hence we have
\begin{equation}\label{vol_dens_esti}
|\mm{det}D\psi(\Th,y)|\leq \rho_{\Th}(y)\norm{D_{\Theta}\psi(\Th,y)} \leq \rho_{\Th}(y) T. 
\end{equation}
With the help of the transformation formula for integrals and using (\ref{vol_dens_esti}) we estimate 
\begin{multline*} 
\int\limits_{\mL}|u|\,dV_{n+1} 
=
\int_{0}^{1}\int\limits_{B_{R_2}}|u\circ \psi(\Th,y)||\mathrm{det}D\psi(\Th,y)|dyd\Th \\ \leq
T\int\limits_{0}^1\int\limits_{B_{R_2}}|u\circ \psi(\Th,y)|\rho_{\Th}(y)dyd\Th
\leq T \sup_{\Th\in[0,1]} \int\limits_{\mathcal H_{\Th}} |u(\Th,\cdot)|\,dV_n,
\end{multline*}
which proves the first inequality. For the term $\frac{d}{d\Th}\int_{\mLT}|u|\,dV_{n+1}$ we find 
\begin{multline*}
\frac{d}{d\Th}\int\limits_{\mLT}|u|\,dV_{n+1}
=
\frac{d}{d\Th}\int\limits_{0}^{\Th}\int\limits_{B_{R_2}}|u\circ \psi(\e,y)||\mathrm{det}D\psi(\e,y)|dyd\e 
\\  
\leq                          
 T\int\limits_{B_{R_2}}|u\circ \psi(\Th,y)|                                               \rho_{\Th}(y)dy 
= 
T\int\limits_{\mHT} |u(\Th,\cdot)|\,dV_n. 
\end{multline*}
\end{proof}

For convenience of the reader we also give a full proof of a Gronwall-type estimate in \cite{BG:07}, Appendix A, Lemma A.3, focusing on explicit expressions for all constants  appearing in the calculation. 

\begin{lemma}\label{multgronwall} 
Let $\mathcal L$ be a standard lens of thickness $T$ and suppose that  $u$ and $f$ are  functions of class $ C([0,T]\times \mathbb R^n)$ such that for all $\Theta\in(0,1]$
\begin{equation}\label{GW_hypothesis}
 \frac{1}{2}\int\limits_{\mathcal H_{\Theta}}|u(\Th,\cdot)|\,dV_n\leq  \int\limits_{\mathcal H_0}|u(0,\cdot)|\,dV_n+\alpha \int\limits_{\mathcal L_{\Theta}} |u|\,dV_{n+1}+ \int\limits_{\mathcal L} |f|\,dV_{n+1} 
\end{equation}
with $\alpha>0$. 
Then we have with $C:=2T\alpha $,
\begin{equation}\label{finerGW_esti}
\frac{1}{2}\int\limits_{\mathcal H_{\Theta}}|u(\Th,\cdot)|\,dV_n \leq e^{C\Th}\Big(\int\limits_{\mathcal H_0}|u(0,\cdot)|\,dV_n +\int\limits_{\mathcal L}|f|\,dV_{n+1}\Big)\quad \forall \Theta\in[0,1]\,.
\end{equation}
\end{lemma}

\begin{proof}
We introduce $v\in C([0,1])$, $v(\Th):=\frac{1}{2T}\int_{\mLT}|u|\,dV_{n+1}$ for $\Th\in(0,1]$ and $v(0):=0$. 
In addition we put $a:=\int_{\mathcal H_0}|u(0,\cdot)|\,dV_n+\int_{\mL}|f|\,dV_{n+1}$. 
From Lemma \ref{lens_int} we know that 
$ 0\leq v^{\prime}(\Th) \leq\frac{1}{2}\int_{\mHT}|u|\,dV_n 
$
for $\Th\in(0,1]$.
Therefore 
$ v^{\prime}(\Th)\leq a+Cv(\Th),
$
where $C:=2T\alpha$.
By integrating the last equation we find
$ v(\Th)\leq a\Th+ C\int_0^{\Th} v(\tau)d\tau
$
and Gronwall's lemma yields 
$
Cv(\Th)\leq   Ce^{C \Th}\int_0^{\Th} ae^{-C\tau}d\tau 
\leq \big(e^{C\Th}-1 \big)a$
for all
$\Th\in[0,1]$.
Expressing the result in terms of $u$ and $f$,  we obtain
\begin{equation}\nonumber
\alpha \int\limits_{\mLT} |u|\,dV_{n+1}\leq \big(e^{C\Th}-1 \big)\Big(\int\limits_{\mathcal H_0}|u(0,\cdot)|\,dV_n+\int\limits_{\mL}|f|\,dV_{n+1}\Big)\quad \forall \Th\in (0,1].
\end{equation} 
Using assumption (\ref{GW_hypothesis}), we obtain (\ref{finerGW_esti}). 
\end{proof}

To a first-order operator $P(t,x;\partial_t,\partial_x)=\partial_t+\sum_{j=1}^n A^j(t,x)\partial_{x_j}+B(t,x)$ we assign its principal symbol $\sigma(t,x;\tau,\xi)=\tau \mathbb I_m +\sum_{j=1}^n A^j(t,x)\xi_j$. If the matrices $A^j$ are Hermitian, then the principal symbol is Hermitian for all directions $(\tau,\xi)\in\mathbb R^{n+1}$. At a point $(t,x)\in\mathbb R^{n+1}$  one may then define the \emph{forward cone} $\Gamma(t,x)$ as the set of all directions $(\tau,\xi)$ where $\sigma(t,x;\tau,\xi)$ is a positive definite matrix. A hypersurface is called \emph{spacelike} (with respect to the principal symbol of $P$) if its normal vector is almost everywhere contained in the forward cone. In the following lemma we will construct a lens whose individual slices are spacelike hypersurfaces with common boundary $\partial \mHT=\{(0,x)|\,|x|=R_2\}$. 

\begin{lemma}\label{smallest_eigenvalue}
Let $(A^j)_{1\leq i\leq n}$ be Hermitian matrices such that $\norm{A_j(t,x)}{\mathrm{op}}\leq C$ for $|x|\geq R_A$. 
Then a  standard lens $\mL$ of thickness $T > 0$ with radii $R_1 \geq R_A$ and $R_2\geq R_1+T(1+2\sqrt{n} C)$ has a unit normal vector field $\nu_{\Th}$ on $\mHT$
(pointing outwards with respect to $\mLT$)
satisfying the inequality
\begin{equation}\label{smallest_ev}
\langle \eta, \sigma(t,x;\nu_{\Th}(t,x)) \eta\rangle
 \geq \frac{1}{2} |\eta|^2 \quad\forall (t,x)\in\mHT\:\:\forall \Th\in(0,1]\:\:\forall \eta\in \mathbb R^m.
\end{equation}
\end{lemma}

\begin{proof}
For a lens of thickness $T$ and radii $R_1, R_2$ the normal vector field on $\mHT$ with $\Th\in(0,1]$ 
 is simply $\nu_{\Th}=(1,0)^T$ for $|x|<R_1$ and
\begin{equation}\label{normal_field_expression}
\nu_{\Th}(t,x)=
\frac{1}{\sqrt{(\Th T)^2+(R_2-R_1)^2}}
\left(\begin{array}{c}R_2-R_1 \\ \Th T\frac{x}{|x|}
\end{array}\right)\!\quad \mathrm{for}\quad R_1<|x|< R_2.
\end{equation}
The inequality in (\ref{smallest_ev}) is obviously satisfied whenever $|x|< R_1$.
For $|x|>R_1\geq R_A$ we find by virtue of (\ref{normal_field_expression}) that
\begin{equation}\nn
\langle \eta, (\nu_{\Th}^0\mathbb I_m+\sum_{j=1}^n \nu_{\Th}^j A^j)\eta\rangle   
\geq  \nu_{\Th}^0 |\eta|^2 - C\sum_{j=1}^n |\nu_{\Th}^j|\,|\eta|^2
\geq \frac{(R_2-R_1-TC\sqrt{n} )}{\sqrt{(\Th T)^2+(R_2-R_1)^2}}|\eta|^2
\end{equation}
and $\frac{(R_2-R_1-TC\sqrt{n} )}{\sqrt{(\Th T)^2+(R_2-R_1)^2}}\geq \frac{1}{2}$ whenever $R_2\geq R_1+T(1+2\sqrt{n} C)$.  
\end{proof}

A first-order partial differential operator $P=\partial_t+\sum_{j=1}^n A^j\partial_{x_j}+B$ with smooth coefficient matrices is called \emph{symmetric hyperbolic} if the matrices $A^j$ and $B$ are uniformly bounded together with all their derivatives and the principal coefficients $A^j$ are Hermitian.  Preparatory for applications to Colombeau theory we perform energy estimates for symmetric hyperbolic operators on standard lenses. It is important to keep explicit expressions for all constants involved   to have precise information on their $\e$-dependence in a generalized setting later on. We provide $L^2$-estimates in two versions, the second of which can be interpreted as the limiting case for lenses with infinite radius.

\begin{lemma}\label{basic_estimate}
Let a symmetric hyperbolic partial differential operator $P=\partial_t +\sum_{j=1}^n A^j\partial_{x_j}+B$ be given, 
where $A^j$ and $B$ are matrices of dimension $m$.
\begin{itemize}
\item[i)]Let $\mL\subseteq [0,T]\times \mathbb R^n$ be a standard lens of thickness $T$ that satisfies inequality (\ref{smallest_ev}) and
put 
$\alpha(\mL) :=\,1\,+\norm{\mm{div}
 A-B-B^{\ast}}{L^{\infty}(\mL)}$,
 where $\mm{div}A=\sum_{j=1}^n\partial_{x_j} A^j$. Here and in the sequel, the $L^\infty$-norm of matrix valued functions is understood as taking the operator norm first and then the supremum over all   $(t,x)\in\mathcal L$.
Then for any $U\in C^{\infty}([0,T]\times \mathbb R^{n})^m$ we have
\begin{equation}\label{final_L2_estimate}
\norm{U}{L^2(\mL)}^2 
\leq 2Te^{2T\alpha(\mL) }\big(\norm{U}{L^2(\mathcal H_0)}^2+\norm{PU}{L^2(\mL)}^2 \big).
\end{equation}
\item[ii)]Denote $\Omega_t:=(0,t)\times \mathbb R^n$ and 
$\beta(t) :=1+\norm{\mm{div} A(t,\cdot)-B(t,\cdot)-B^{\ast}(t,\cdot)}{L^{\infty}(\Omega_T)}$.
Then for any $U\in C^1([0,T],L^2(\mathbb R^{n}))^m\cap  C^0([0,T],H^{1}(\mathbb R^n))^m$ the following estimate holds for all $t\in[0,T]$:
\begin{equation}\label{special_L2_estimate}
\norm{U(t,\cdot)}{L^2(\mathbb R^n)}^2\leq e^{\int\limits_0^t\beta(s)ds}\big(\norm{U(0,\cdot)}{L^2(\mathbb R^n)}^2+\norm{PU}{L^2(\Omega_t)}^2 \big).
\end{equation}
\end{itemize}
\end{lemma}

\begin{proof}
 Applying the operator $P$ to an arbitrary $U\in C^{\infty}([0,T]\times \mathbb R^{n})^m$ we may write
\begin{equation}\label{esti_1}
\langle\partial_t U,U\rangle +\sum_{j=1}^n\langle A^j\partial_{x_j}U,U\rangle
+\langle BU,U\rangle\,=\,\langle PU,U\rangle\,.
\end{equation}
A short calculation shows that
\begin{equation}
\nn
2\mathrm{Re}\langle \partial_tU,U\rangle+2\mathrm{Re}\sum_{j=1}^n\langle A^j\partial_{x_j}U,U\rangle
=
\partial_t\Norm{U}^2
+
\sum_{j=1}^n\partial_{x_j}\langle A^jU,U\rangle
-
\langle(\mm{div}A)U,U\rangle.
\end{equation}
Thus, taking two times the real part of (\ref{esti_1}) we conclude 
\begin{equation}
\label{pointwise_eq}
 \partial_t\Norm{U}^2 + \sum_{i=1}^{n}\partial_{x_j}\langle A^jU,U\rangle
-
\langle(\mm{div}A)U,U\rangle+\,2\mathrm{Re}\langle BU,U\rangle\,=\,2\mathrm{Re}\langle PU,U\rangle
\,.
\end{equation} 
To prove i), we rewrite (\ref{pointwise_eq}) as 
$$ \mathrm{div}(\Norm{U}^2,\langle A_1 U,U\rangle,...,\langle A_nU,U\rangle)=
\langle (\mm{div}A)U,U \rangle -\langle (B+B^{\ast})U,U\rangle+2\mm{Re}\langle P U,U\rangle
\,.
$$
Integrating over a partial lens $\mathcal L_{\Th}\subseteq\mathbb R^{n+1}$, the divergence theorem yields
\begin{equation}
\nn
\int\limits_{\partial \mLT}\!\!\!\big(\nu_{\Th}^0 \Norm{U}^2+\sum_{j=1}^n \nu_{\Th}^j\langle A^jU,U\rangle \big)\,dS
=\!\!\int\limits_{\mLT} \langle (\mm{div}A -B-B^{\ast})U,U\rangle\,dV+2\mm{Re}\!\!\int\limits_{\mLT} \langle PU,U\rangle\, dV
\,.
\end{equation}
where $\nu_{\Th}$ is the unit normal vector field on $\partial \mLT$, assumed to point outwards with respect to $\mLT$. 
Since $\nu_0=(-1,0)^T$, we conclude from inequality (\ref{smallest_ev}) that 
$$ \int\limits_{\partial\mLT}\big(\nu_{\Th}^0 \Norm{U}^2+\sum_{j=1}^n \nu_{\Th}^j\langle A^jU,U\rangle \big)\,dS\geq \frac{1}{2}\norm{U(\Th,\cdot)}{L^2(\mHT)}^2-\norm{U(0,\cdot)}{L^2(\mathcal H_0)}^2. 
  $$ 
Hence  for all $\Th\in(0,1]$ the term $\frac{1}{2}\norm{U(\Th,\cdot)}{L^2(\mHT)}^2$ is bounded by
\begin{equation}
\nn
\norm{U(0,\cdot)}{L^2(\mathcal H_0)}^2
+
\int\limits_{\mLT} \langle (\mm{div}A -B-B^{\ast})U,U\rangle\,dV+2\mm{Re}\int\limits_{\mLT} \langle PU,U\rangle \,dV
\,.
\end{equation}
The terms on the right-hand side can be estimated with the help of the Cauchy-Schwarz inequality, leading to
\begin{equation}
\nn
\frac{1}{2} \norm{U}{L^2(\mHT)}^2 
\leq    \norm{U(0,\cdot)}{L^2(\mathcal H_0)}^2+\alpha(\mL) \norm{U}{L^2(\mLT)}^2 +\norm{PU}{L^2(\mLT)}^2
\end{equation}
 and Lemmas \ref{multgronwall} and \ref{lens_int}  imply 
\begin{equation}
\nn
\norm{U}{L^2(\mL)}^2 
\leq 
2T\,e^{2T\alpha(\mL) }\big(\norm{U(0,\cdot)}{L^2(\mathcal H_0)}^2+\norm{PU}{L^2(\mL)}^2 \big).
\end{equation}
To prove ii), we  integrate in (\ref{pointwise_eq}) over the spatial domain  $\mathbb R^n$, leading to 
\begin{multline*} \frac{d}{dt}\norm{U(t,\cdot)}{L^2(\mathbb R^n)}^2=\\
-\sum\limits_{j=1}^n\:\int\limits_{\mathbb R^n} \partial_{x_j}\langle A^j U,U \rangle\,dV_n
+\int\limits_{\mathbb R^n}\langle (\mm{div}A-B-B^{\ast})U,U\rangle\,dV_n\,
+\,2\mathrm{Re}\int\limits_{\mathbb R^n}\langle PU,U\rangle\,dV_n.
\end{multline*}
After integration with respect to the time variable  between $0$ and $t$ we find 
\begin{equation}\label{gronwall_entrance}
\norm{U(t,\cdot)}{L^2(\mathbb R^n)}^2
\leq \norm{U(0,\cdot)}{L^2(\mathbb R^n)}^2 +
\int\limits_0^t \beta(s)\norm{U(s,\cdot)}{L^2(\mathbb R^n)}^2ds
+ \int_0^t\norm{PU(s,\cdot)}{L^2(\mathbb R^n)}^2ds 
\end{equation} 
where
we have used that $\int_{\mathbb R^n}\partial_{x_j}\langle A^j(s,\cdot)U(s,\cdot),U(s,\cdot) \rangle dV=0$ for all $j=1,...,n$ and for all $s\in[0,T]$ since $x\mapsto U(s,x)$ belongs to $H^1(\mathbb R^n)^m$ and the latter possesses 
$ C^{\infty}_c(\mathbb R^n)^m$ as a dense subspace. 
Employing Gronwall's lemma we turn (\ref{gronwall_entrance}) into the desired estimate
\begin{equation*}
\norm{U(t,\cdot)}{L^2(\mathbb R^n)}^2 
\leq  e^{\int\limits_0^t\beta(s)ds}\big(\norm{U(0,\cdot)}{L^2(\mathbb R^n)}^2+\norm{PU}{L^2(\Omega_t)}^2 \big). 
\end{equation*}
\end{proof}

\section{Generalized solutions to the Cauchy problem}

Having all necessary prerequisites at hand we draw our attention to the initial value problem (\ref{ivp}-\ref{ivp_ic}) on the space-time domain $\Omega_T:=(0,T)\times \mathbb R^n$. We will establish three statements of  existence and uniqueness, each using different spaces of initial data and right-hand side. Working with a smaller space in this respect allows to relax the asymptotic conditions on the coefficient matrices $A^j$ and $B$.

The formulation of the theorems requires some notions from Colombeau theory, we want to briefly review. A generalized function $u\in\mathcal G(\Omega)$ is called of $L^{\infty}$-\emph{type} if it has a $C^{\infty}$-moderate representative $(u_{\e})_{\e}$ such that $\norm{u_{\e}}{L^{\infty}(\Omega)}=O(\e^{-m})$ as $\e\rightarrow 0$. It is called \emph{locally of logarithmic growth} or \emph{locally log-type}, if it has a $C^{\infty}$-moderate representative $(u_{\e})_{\e}$ such that for all $K\subset\subset \Omega$,   $\norm{u_{\e}}{L^{\infty}(K)}=O(\mm{log}(1/\e))$ as $\e\rightarrow 0$  (cf. Definition 1.1 in \cite{O:88}). It is called of $L^{\infty}$-\emph{log-type} if $\norm{u_{\e}}{L^{\infty}(\Omega)}=O(\mm{log}(1/ \e))$ as $\e \rightarrow 0$ (cf. Definition 1.5.1 in \cite{GKOS:01}). A matrix 
$A\in M_m(\mathcal G(\Omega))$  
is called \emph{Hermitian}, if it has a  Hermitian representative $(A_{\e})_{\e}$, i.e. $A_{\e}$ is Hermitian for all $\e< \e_0$ (cf. Lemma 4.3 in \cite{Mayerhofer:08}).

We call a partial differential operator $P=\partial_t+\sum_{j=1}^n A^j \partial_{x_j} +B$ with   Colombeau generalized  coefficient matrices  \emph{symmetric hyperbolic},  if all matrices $A^j$ are Hermitian and the entries of $A^j$ and $B$  are of $L^{\infty}$-type together with all their derivatives, i.e. $A^j,B\in  M_m(\mathcal G_{L^{\infty}}(\Omega)) $. This ensures that there exists $\e_0>0$ and representatives  $(A^j_{\e})_{\e}$ and $(B_{\e})_{\e}$  such that $P_{\e}=\partial_t+\sum_{j=1}^n A^j_{\e}\partial_{x_j} + B_{\e}$ is a classical symmetric hyperbolic operator for all $\e< \e_0$.  
The corresponding family of smooth solutions to the classical Cauchy problem for fixed $\e$ represents a candidate for the generalized solution. Yet some  additional asymptotic growth conditions in $\e$  have to be imposed on the coefficients to obtain a moderate family of solutions. In particular,   
certain log-type  conditions on the coefficient matrices are essential
 in order to use a Gronwall-type argument in the proof (cf. \cite{Hoermann:10,GH:04,LO:91,O:88,O:89}).

The first theorem allows for the most general initial data and right-hand side, but requires the principal coefficients $A^j_{\e}(t,x)$ to be bounded uniformly in $\e$ and $(t,x)$   for large $|x|$. 

\begin{theorem}\label{main_theorem}
The initial value problem for a  symmetric hyperbolic operator with Colombeau generalized coefficients, 
\begin{eqnarray}
\label{givp_eq}
\partial_t U +\sum_{j=1}^n A^j\partial_{x_j}U+BU \!\!&=&\!\! F\qquad \text{on }\Omega_T\\
\label{givp_in}
 U(0,x)\!\!&=&\!\!G(x),
\end{eqnarray}
has a unique solution $U\in \mathcal G(\Omega_T)^m$, if
\begin{itemize}
\item[i)] initial data $G\in\mathcal G(\mathbb R^{n})^m$ and right-hand side $F\in\mathcal G([0,T]\times\mathbb R^n)^m$,
\item[ii)] all spatial derivatives $\partial_{x_i}A^j$ as well as the Hermitian part of $B$ are locally of log-type,
\item[iii)] there exists $R_A>0$ such that $\norm{A^j_{\e}(t,x)}{\mathrm{op}}=O(1)$  on $(0,T)\times \{x\in\mathbb R^n|\,|x|>R_A \}$ as $\e\rightarrow 0$. 
\end{itemize}
\end{theorem}

\begin{proof}
We pick Hermitian representatives $(A^j_{\e})_{\e}$ and representatives of $B$, $F$ and $G$. There exists $\e_0>0$ such that $\forall \e< \e_0$, the initial value problem
\begin{eqnarray}\label{ivp_epsi}
\partial_t U_{\e} +\sum_{j=1}^n A^j_{\e}\partial_{x_j}U_{\e}+B_{\e}U_{\e} \!\!&=&\!\! F_{\e}\\
\label{ivp_epsi_in}
U_{\e}|_{t=0} \!\!&=&\!\! G_{\e}.
\end{eqnarray}
has a unique solution $U_{\e}\in  C^{\infty}(\overline \Omega_T)^m$ (cf. Theorem 2.12 in \cite{BG:07}).
We claim that the equivalence class $U=[(U_{\e})_{\e}]$ is the unique Colombeau solution. Hence we must show that the net $(U_{\e})_{\e}$ is moderate and that   
negligible variations of the coefficients and the data yield the same solution. \newline
Let $K\subset\subset \Omega_T$ and assume $K\subseteq(0,T)\times \{x\in\mathbb R^n|\, |x|\leq R_K \}$. Choose $R_1>\mm{max}(R_A,R_K)$ and 
$R_2\geq R_1+T(1+2\sqrt{n}C)$. Then a standard lens $\mL$ of thickness $T$, inner radius $R_1$ and outer radius $R_2$ will contain $K$ and satisfy inequality (\ref{smallest_ev}) by Lemma \ref{smallest_eigenvalue}. Thus we may apply  the energy
estimate (\ref{final_L2_estimate}) $\e$-wise and obtain 
\begin{equation}\label{L2_estimate_espilon}
\begin{array}{ccccc}
\norm{U_{\e}}{L^2(\mL)}^2 
&\leq& 2T\,e^{2T\alpha_{\e}(\mathcal L)}\big(\norm{G_{\e}}{L^2(\mathcal H_0)}^2+\norm{F_{\e}}{L^2(\mL)}^2 \big)
&=& O(\e^{-m})
\end{array}
\end{equation}
as $\e\rightarrow 0$ for some $m\in\mathbb N_0$, since the norms of the  data  grow only like some inverse power of $\e$  and $\alpha_{\e}(\mathcal L) = 1+ \norm{\mm{div}A_{\e}-B_{\e}-B^{\ast}_{\e}}{L^{\infty}(\mL)}= O(\mm{log}(1/\e))$  as $\e\rightarrow 0$ by assumption ii).
We attempt to show that all derivatives of $U_{\e}$ satisfy a similar estimate. For this purpose we introduce some convenient notations. For $A\in  M_m( C^{\infty}(\overline\Omega_T))$ and $U\in C^{\infty}(\Omega_T)^m$ we define
$$
\begin{array}{rclrcl}
\nabt{1}A 
&:=&
\mm{diag}(\partial_{x_1}A,...,\partial_{x_n}A)\qquad 
& \nab{1} U
&:=&
(\partial_{x_1}U,...,\partial_{x_n}U)^T\\
\nabt{r+1} A
&:=& \widetilde\nabla^{1}\nabt{r} A\qquad
& \nab{r+1} U 
&:=& \nab{1} \nab{r} U\\
\sigmt{r} A 
&:=& \mm{diag(A,...,A)}_{n^r}\qquad
& \sigm{r} U
&:=& 
(U,...,U)^T_{n^r}
 \,,
\end{array}
$$
For example, $\nabt{r} A$ is a blockdiagonal matrix built from all spatial derivatives $\partial_x^{\alpha}\!A$ 
of length $|\alpha|=r$. Similarly $\sigmt{r} A$ is a blockdiagonal matrix whose blocks are just $n^r$ copies of $A$ itself.

\begin{claim}
The vector $\nabla^rU_{\e}$ satisfies an equation of the form
\begin{equation}\label{higherorderesti}
\partial_t \nabla^r U_{\e}+\sum_{j=1}^n \sigmt r A^j_{\e}\partial_{x_j}\!\nabla^rU_{\e}+\widetilde B^{r}_{\e}\nabla^rU_{\e}=Q^{r-1}_{\e} \sigm r U_{\e} +\nabla^r F_{\e}
\end{equation}
where 
$\widetilde{B}^1_{\e}=\sigmt 1 B_{\e}+(\partial_{x_i}A^j_{\e})_{1\leq i,j\leq n}$, 
$\widetilde B^{r+1}_{\e}=\sigmt 1\widetilde B^{r}_{\e}+(\partial_{x_i}\sigmt r A^j_{\e})_{1\leq i,j\leq n}$ for 
$ r\geq 1$, 
and $Q^{r-1}_{\e}$ is  a purely spatial partial differential operator of order $r-1$ with coefficients depending linearly on spatial derivates of $A^j_{\e}$ and $B_{\e}$ up to order $r$.
\end{claim}

We present the case $r=1$ in detail and proceed by induction.
Differentiating (\ref{ivp_epsi}) with respect to $x_k$ yields
$$ 
\partial_t\partial_{x_k}U_{\e}+ \sum_{j=1}^n A^j_{\e}\partial_{x_j}\partial_{x_k}U_{\e} +\sum_{j=1}^n \partial_{x_k}A^j_{\e}\partial_{x_j}U_{\e}+B_{\e}\partial_{x_k}U_{\e}+\partial_{x_k}B_{\e}U_{\e}=\partial_{x_k}F_{\e}
$$ 
One would like to read this as an equation for $\partial_{x_k}U_{\e}$, but since the equations are coupled one has to consider the system
\begin{equation}\nn \partial_t \nabla^1 U_{\e} + \sum_{j=1}^n 
\sigmt 1 A^j_{\e}\partial_{x_j}\nabla^1 U_{\e}
+\sum_{j=1}^n\nabt 1 A^j_{\e} \partial_{x_j}\sigm 1 U_{\e} + \sigmt 1 B_{\e}\,\nabla^1 U_{\e}  = -\nabt 1 B_{\e}\,\sigm 1 U_{\e}+\nabla^1 F_{\e}
\end{equation}
There are no derivatives of $U_{\e}$ on the right-hand side and the only term that does not fit into our concept on the left-hand side can be rewritten in the following way,
$$  \sum_{j=1}^n\nabt 1 A^j_{\e} \partial_{x_j}\sigm 1 U_{\e} = (\partial_{x_i}A^j_{\e})_{1\leq i,j\leq n}\nabla^1 U_{\e}
$$
so that $\nabla^1U_{\e}$ satisfies the system
$$ \partial_t \nabla^1 U_{\e}  + \sum_{j=1}^n \sigmt 1 A^j_{\e}\partial_{x_j}\nabla^1 U_{\e}
+\widetilde B^1_{\e} \nabla^1U_{\e}  = Q^0_{\e}\sigm 1 U_{\e}+\nabla^1 F_{\e}
$$
where
$ \widetilde B^1_{\e}  =  \sigmt 1B_{\e}+(\partial_{x_i}A^j_{\e})_{1\leq i,j\leq n}$ and 
$ Q^0_{\e}  =  -\nabt 1 B_{\e}$.
We proceed by induction with respect to the differentiation index $r$. Applying $\partial_{x_k}$ to the induction hypothesis (\ref{higherorderesti}) we find
\begin{multline*}
\partial_t\partial_{x_k}\!\nabla^{r}U_{\e}     +       \sum_{j=1}^n \! \sigmt r A^j_{\e}\,\partial_{x_j}\partial_{x_k}\!\nabla^{r}U_{\e} 
     +        \sum_{j=1}^n \partial_{x_k}\sigmt r A^j_{\e}\partial_{x_j}\!\nabla^rU_{\e} +
\widetilde B^{r}_{\e}\partial_{x_k}\!\nabla^{r}U_{\e}  \\
      =            -\partial_{x_k}\widetilde B^{r}_{\e}\nabla^r U_{\e}+ \partial_{x_k}(Q^{r-1}_{\e}\sigm r U_{\e})+\partial_{x_k}\!\nabla^{r} F_{\e}\,.
\end{multline*}
Just like in the case $r=1$ we try to write these $k$ systems as one big system. It is easy to see that the right-hand side can be written as $Q^{r}_{\e}\sigm{r+1}U_{\e}+\nabla^{r+1}F_{\e}$ with $Q^{r}_{\e}$ a purely spatial partial differential operator of order $r$ with coefficients depending linearly on spatial derivates of $A^j_{\e}$ and $B_{\e}$ up to order $r+1$. Furthermore we can rewrite the lower-order terms on the left-hand side as $\widetilde B^{r+1}_{\e}\nabla^{r+1}U_{\e}$, which finally leads to 
\begin{equation}
\nn
\partial_t\nabla^{r+1}U_{\e} +  \sum_{j=1}^n\!\sigmt{r+1}A^j_{\e}\,\partial_{x_j}\!\nabla^{r+1}U_{\e} +\widetilde B^{r+1}_{\e} \nabla^{r+1}U_{\e}
=Q^r_{\e}\sigm{r+1}U_{\e}+\nabla^{r+1} F_{\e}\,,
\end{equation}
where indeed
$\widetilde B^{r+1}_{\e}     =   \sigm 1\widetilde B^{r}_{\e} +(\partial_{x_i}\sigmt rA^j)_{1\leq i,j\leq n}.$
\newline\newline
Since $\norm{\sigmt r A^j_{\e}}{\mathrm{op}}=\norm{A^j_{\e}}{\mathrm{op}}$, the estimate (\ref{smallest_ev}) is also valid for the symmetric hyperbolic operator $P^r=\mathbb I_{n^rm}\partial_t+\sum_{j=1}^n \sigmt r A^j\partial_{x_j}+\widetilde B^r$. By (\ref{final_L2_estimate}) in Lemma \ref{basic_estimate} i)   we therefore have
\begin{equation}\label{higherorder_estimate}
\norm{\nab rU_{\e}}{L^2(\mL)}^2 
\leq 2T\,e^{2T\alpha^r_{\e}(\mathcal L) }\big( \norm{\nab r G_{\e}}{L^2(\mathcal H_0)}^2+
\norm{Q^{r-1}_{\e}\sigm{r}U_{\e}+\nabla^{r} F_{\e}}{L^2(\mL)}^2 \big)
\end{equation}
where 
$$
\begin{array}{ccccc} \nn
\alpha_{\e}^r(\mL) &:=& 1+\norm{\mm{div}\,\sigmt r A_{\e}-\widetilde B^r_{\e}-(\widetilde B^r_{\e})^{\ast}}{L^{\infty}(\mL)}&=& O(\mm{log}(1/\e))\quad\mm{as}\quad\e\rightarrow 0,
\end{array}
$$
since the Hermitian part of $\widetilde B^r_{\e}$ can be constructed from the Hermitian part of $B_{\e}$ as well as first-order derivatives  $\partial_{x_i}A^j_{\e}$.
From the fact that $\norm{U_{\e}}{L^2(\mL)}^2=O(\e^{-m})$ and by iterative application of (\ref{higherorder_estimate}) for $r=1,2,3,...$ we  conclude that 
\begin{equation}\label{higherorder_asy_esti}
\forall  r\in\mathbb N_0\,\exists m\in \mathbb N_0:\, \norm{\nab r U_{\e}}{L^2(\mL)}^2=O(\e^{-m})\quad\mm{as}\quad\e\rightarrow 0.
\end{equation}
It remains to show this asymptotic estimate for derivatives involving also the $t$-coordinate. 
\begin{claim}
All mixed derivatives $\partial_t^l\nabla^rU_{\e}$ satisfy an equation of the form
\begin{equation}\label{higherorderestitime}
\partial_t^l\nabla^r U_{\e}=R^{r,l}_{\e}\sigm rU_{\e} +\partial_t^{l-1}\nabla^r F_{\e}
\end{equation}
where $R^{r,l}_{\e}$ is a linear partial differential operator of order $r+l$ involving $t$-derivatives only up to order $l-1$. Moreover, the coefficients of $R^{r,l}_{\e}$ are linear combinations of spatial derivatives of $A^j_{\e}$ and $B_{\e}$ up to order $r+1$ and time derivatives of $A^j_{\e}$ and $B_{\e}$ up to order $l-1$.
\end{claim}

The case $l=1$ follows immediately from (\ref{higherorderesti}) by putting
$$ R^{r,1}_{\e}\sigm rU_{\e}:= Q^{r-1}_{\e}\sigm rU_{\e}-\sum_{j=1}^n \sigmt r A^j_{\e}\partial_{x_j}\nabla^rU_{\e} - \widetilde B^r_{\e}\nabla^rU_{\e}
$$
since $Q^{r-1}_{\e}$ is a purely spatial operator of order $r-1$.
Applying the operator $\partial_t$ to the induction hypothesis (\ref{higherorderestitime}) gives
$$\partial_t^{l+1}\nabla^r U_{\e}= \partial_t(R^{r,l}_{\e} \sigm rU_{\e})+\partial_t^{l}\nabla^r\! F_{\e}
$$
and $R^{r,l+1}_{\e}\sigm r U_{\e}:=\partial_t (R^{r,l}_{\e}\sigm rU_{\e})$ is of course an operator of order $r+l+1$ with time derivatives only up to order $l$. As an obvious implication of the Leibniz rule, its coefficients are linear combinations of derivatives of $A^j_{\e}$ and $B_{\e}$ with spatial derivatives of order $r+1$ at most, since the coefficients of the operator $Q^{r-1}_{\e}$ depend (only) on spatial derivatives of $A^j_{\e}$ and $B_{\e}$ up to order $r+1$. 
\newline\newline
Successively making use of (\ref{higherorderestitime}) for $l=1,2,3,...$ yields in combination with (\ref{higherorder_estimate}) that
\begin{equation}
\forall l,r\in\mathbb N_0\,\exists m\in\mathbb N_0:\,\norm{\partial_t^l\nab r U_{\e}}{L^2(\mL)}=O(\e^{-m})\quad\mm{as}\quad\e\rightarrow 0
\end{equation}
and by the Sobolev embedding theorem on domains with locally Lipschitz boundary (cf. Theorem 4.12 Part II in \cite{AF:03}) and the fact that $K\subseteq \mL$ this implies 
\begin{equation}\label{final_Linf_estimate}
\forall \alpha\in\mathbb N_0^{n+1}\,\exists m\in\mathbb N_0:\,\norm{\partial^{\alpha} U_{\e}}{L^{\infty}(K)}=O(\e^{-m})\quad\mm{as}\quad\e\rightarrow 0,
\end{equation}
i.e. the class $[(U_{\e})_{\e}]$ is moderate, since $K\subset\subset \Omega_T$ was arbitrary. For the uniqueness part we choose negligible nets $(\overline F_{\e})_{\e}$ and $(\overline G_{\e})_{\e}$ to represent right-hand side and initial data.  
From the energy estimates (\ref{L2_estimate_espilon}), (\ref{higherorder_estimate}), and equation (\ref{higherorderestitime}) it is then easy to see that the corresponding solution $[(\overline U_{\e})_{\e}]$ will also be negligible. By Theorem 1.2.3 in \cite{GKOS:01} it actually suffices to show the negligibility estimate for the zeroth derivative only, i.e. in terms of $L^2$-estimates for all derivatives of $\overline U_{\e}$ with order $\leq\lceil(n+1)/2\rceil$.  
\end{proof}

If the $O(1)$-condition on $A^j_{\e}(t,x)$ for large $|x|$ is dropped, one cannot work with the same lens for all values of $\e$ anymore, but has to use an $\e$-indexed family of standard lenses. However, since the growth of the volumes of these lenses is under control, we can still keep many aspects of the solvability result by subjecting  initial data and right-hand side to the stricter growth conditions of the space $\mathcal G_{L^{\infty}}$. 

\begin{theorem}\label{second_theorem}
Consider  the  alternative conditions
\begin{itemize}
\item[i')] initial data  $G\in\mathcal G_{L^{\infty}}(\mathbb{R}^n)^m$ and right-hand side $F\in\mathcal G_{L^{\infty}}(\Omega_T)^m$,
\item[ii')] all spatial derivatives $\partial_{x_i}A^j$ as well as the Hermitian part of $B$ are of $L^{\infty}$-log-type.
\end{itemize}
Then there exists a unique $U\in\mathcal G(\Omega_T)^m$ satisfying Equation (\ref{givp_eq}) and such that (\ref{givp_in}) holds in the following sense: $U\!\!\mid_{t = 0}$ is equal to the image of $G$ under the canonical map $\mathcal G_{L^{\infty}}(\mathbb{R}^n)^m \to \mathcal G(\mathbb{R}^n)^m$.
\end{theorem}

\begin{proof}
The loss of condition iii) and the alternative version of i) have no effect on the applicability of Theorem 2.12 in \cite{BG:07} for fixed $\e< \e_0$. So we still get a solution candidate $[(U_{\e})_{\e}]$ from the $\e$-wise construction of a family $(U_{\e})_{\e}$. After choosing a compact set $K\subset\subset\Omega_T$, we again aim at building a standard lens around it such that (\ref{smallest_ev}) in Lemma \ref{smallest_eigenvalue} is satisfied. Assuming that $K\subseteq(0,T)\times \{x\in\mathbb R^n|\, |x|\leq R_K \} $, we fix its thickness $T$ and its inner radius $R_1>R_K$.
 Since we now have an $\e$-dependent bound $\norm{A^j_{\e}}{L^{\infty}}\leq C\e^{-m}$, we choose a family of radii 
$R_{2\,\e}= R_1+T(1+2\sqrt{n}C\e^{-m})$,
thereby obtaining a family of standard lenses $(\mL_{\e})_{\e}$ with outer radii $R_{2\,\e}=O(\e^{-m})$ as $\e\rightarrow 0$, all containing the compact set $K$. We put $\alpha'_{\e} = 1+ \norm{\mm{div}A_{\e}-B_{\e}-B^{\ast}_{\e}}{L^{\infty}(\Omega_T)}$, i.e. the supremum taken over the whole domain $\Omega_T$. Employing estimate (\ref{L2_estimate_espilon}) $\e$-wise for the lens $\mathcal L_{\e}$ leads to
\begin{multline*}
\norm{U_{\e}}{L^2(\mL_{\e})}^2 
\leq 2T\,e^{2T\alpha_{\e}'}\big(\norm{G_{\e}}{L^2(\mathcal H_{0\,\e})}^2+\norm{F_{\e}}{L^2(\mL_{\e})}^2 \big)
\\ \leq
2T\,e^{2T\alpha'_{\e}}\big(\mm{Vol}_n(\mathcal H_{0\,\e})\norm{G_{\e}}{L^{\infty}(\mathbb R^n)}^2+\mm{Vol}_{n+1}(\mL_{\e})\norm{F_{\e}}{L^{\infty}(\Omega_T)}^2 \big)=O(\e^{-m})\end{multline*}
for some $m\in\mathbb N_0$,
since both $\mm{Vol}_n(\mathcal H_{0\,\e})$ and $\mm{Vol}_{n+1}(\mL_{\e})$ grow only like some inverse power of $\e$ as $\e\rightarrow 0$ and $\alpha_{\e}'=O(\mm{log}(1/ \e))$. 
Analogously using the higher order energy estimate (\ref{higherorder_estimate}) yields
\begin{multline*}
\norm{\nab rU_{\e}}{L^2(\mL_{\e})}^2 
 \leq  2T\,e^{2T\alpha'^{\,r}_{\e} }\big( \norm{\nab r G_{\e}}{L^2(\mathcal H_{0\,\e})}^2+
\norm{Q^{r-1}_{\e}\sigm{r}U_{\e}+\nabla^{r+1} F_{\e}}{L^2(\mL_{\e})}^2 \big)\\
 \leq
4T\,e^{2T\alpha'^{\,r}_{\e} }\big(\mm{Vol}_n(\mathcal H_{0\,\e}) \norm{\nab r G_{\e}}{L^{\infty}(\mathbb R^n)}^2+
\norm{Q^{r-1}_{\e}\sigm{r}U_{\e}}{L^2(\mL_{\e})}^2 \\
+\mm{Vol}_{n+1}(\mL_{\e})\norm{\nabla^{r+1} F_{\e}}{L^{\infty}(\Omega_T)} \big)
= 
O(\e^{-m})\quad\mm{as}\quad\e\rightarrow 0
\end{multline*}
as the term $\norm{Q^{r-1}_{\e}\sigm{r}U_{\e}}{L^2(\mL_{\e})}^2$ can be estimated via pulling out $L^{\infty}$-norms of derivatives of the coefficients   and $\alpha'^{\,r}_{\e}:=1+\norm{\mm{div}\,\sigmt r A_{\e}-\widetilde B^r_{\e}-(\widetilde B^r_{\e})^{\ast}}{L^{\infty}}=O(\mm{log}(1/\e))$. With the help of (\ref{higherorderestitime}) it is then easy to see that 
$$ \forall l,r\in\mathbb N_0\,\exists m\in\mathbb N_0:\,\norm{\partial_t^l\nab r U_{\e}}{L^2(\mL_{\e})}=O(\e^{-m})\quad(\e\rightarrow 0).
$$ 
Since $K\subseteq \mL_{\e}$ for all $\e< \e_0$ and by the Sobolev embedding theorem we conclude that for all $\alpha\in\mathbb N_0^{n+1}$ there exists $m\in\mathbb N_0$ such that $\norm{\partial^{\alpha} U_{\e}}{L^{\infty}(K)}=O(\e^{-m})$ as $\e\rightarrow 0$.
The uniqueness part is completely analogous to the corresponding part in the proof of Theorem \ref{main_theorem}.
\end{proof}

Theorem \ref{second_theorem} is applicable even for coefficients and data which are both associated to  highly singular and periodic  distributions. Denoting the delta distribution at $y$ by $\delta_y=\delta(\,\cdot\,-y)$, it is  possible  to consider, e.g., principal coefficients $A^j(t,x)\approx \sum_{\kappa\in \mathbb N_0^m}\delta_{l\kappa}(x)\widetilde A^j(t)$  and initial data $G(x)\approx \sum_{\kappa\in \mathbb N_0^m}\delta_{q\kappa}(x) $ with real numbers $l,q>0$,  
representing  $m$-dimensional lattices of Dirac measures with lattice constants $1/l$ and $1/q$, respectively. 
\begin{remark}
The conditions in Theorem \ref{second_theorem} allow for infinite propagation speed near spatial infinity as $\e\rightarrow 0$. In general this may cause non-uniqueness of solutions, see Example 17.1 in \cite{O:92}. 
Using the space $\mathcal G_{L^{\infty}}$ for initial data and right-hand side avoids nonuniqueness, yet null ``solutions'' with non-vanishing initial data still exist. In fact, any initial data $G$ in the kernel of the canonical map $\mathcal G_{L^{\infty}}(\mathbb{R}^n)^m \to \mathcal G(\mathbb{R}^n)^m$ yield $U = 0$.\footnote{Consider the scalar equation $\partial_t u+\frac{1}{\e}\partial_x u=0$ with $u|_{t=0}=\varphi$, where $\varphi \in \mathcal D(\mathbb R)$ and $\varphi(0)=1$. Then the unique Colombeau solution is the class $[(\varphi(x-\frac{1}{\e}t))_{\e}]=0$ in $\mathcal G((0,T)\times \mathbb R)$ since for all $K\subset\subset (0,T)\times \mathbb R$  we have $\mm{supp}(\varphi(x-\frac{1}{\e}t)\cap K=\emptyset$ for $\e$ small enough.  }    
\end{remark}
Initial data and right-hand side decaying at spatial infinity ($|x|\to\infty$) make it possible to relax the conditions on $A^j$ and $B$ even a bit further. More precisely, the required asymptotic behavior of $A^j_{\e}$ and $B_{\e}+B^{\ast}_{\e}$ can be made less restrictive with respect to the time variable. 
We use mixed norms 
$
\norm{A}{L^{1,\infty}(\Omega_T)}:=\int_0^T \norm{A(s,\cdot)}{L^{\infty}(\mathbb R^n)}ds$ for any $A\in  M_m(C_b^{\infty}(\overline\Omega_T)). 
$
We say that an element $A\in M_m(\mathcal G_{L^{\infty}}(\Omega_T))$ is of $L^{1,\infty}$-log-type, if it has a representative $(A_{\e})_{\e}$ such that $\norm{A_{\e}}{L^{1,\infty}(\Omega_T)} = O(\mm{log}(1/\e))$ as $\e\rightarrow 0$. A similar norm was introduced 
in Definition 2.1 in \cite{CO:90}.
\begin{theorem}\label{third_theorem}
In the initial value problem (\ref{givp_eq}-\ref{givp_in}) assume that 
\begin{itemize}
\item[i'')] initial data $G\in\mathcal G_{L^{2}}(\mathbb R^n)^m$ and right-hand side $F\in\mathcal G_{L^{2}}(\Omega_T)^m$,
\item[ii'')] all $\partial_{x_i}A^j$ as well as the Hermitian part of $B$ are of $L^{1,\infty}$-log-type.
\end{itemize}
Then there exists a unique solution $U\in\mathcal G_{L^2}(\Omega_T)^m$ to the initial value problem (\ref{givp_eq}-\ref{givp_in}).
\end{theorem}

\begin{proof}
Fixing Hermitian representatives of $A^j$ and representatives of $B$, $F$ and $G$, we may use Theorem 2.6 in \cite{BG:07} to provide solutions $U_{\e}\in C^{\infty}([0,T],H^{\infty}(\mathbb R^n))^m$ to the classical initial value problem for each $\e< \e_0$. To show moderateness, we plug $U_{\e}$ into the energy estimate (\ref{special_L2_estimate}) and find
\begin{multline}\label{special_esti_beta}
\norm{U_{\e}}{L^2(\Omega_T)}^2\leq T\sup\limits_{0\leq t\leq T}\norm{U_{\e}(t,\cdot)}{L^2(\mathbb R^n)}^2 \\
\leq
T e^{\beta_{\e}}\big(\norm{G_{\e}}{L^2(\mathbb R^n)}^2+\norm{F_{\e}}{L^2(\Omega_T)}^2 \big)=O(\e^{-m}),
\end{multline}
where $\beta_{\e}:=\norm{\mm{div}A_{\e}-B_{\e}-B^{\ast}_{\e}}{L^{1,\infty}(\Omega_T)}=O(\mm{log}(1/\e))$.
In Claim 1 in the proof of Theorem \ref{main_theorem} it has been shown that the vector $\nabla^rU_{\e}$ satisfies an equation of the form
\begin{equation}
\partial_t \nabla^r U_{\e}+\sum_{j=1}^n \sigmt r A^j_{\e}\partial_{x_j}\!\nabla^rU_{\e}+\widetilde B^{r}_{\e}\nabla^rU_{\e}=Q^{r-1}_{\e} \sigm r U_{\e} +\nabla^r F_{\e}
\end{equation}
where $Q^{r-1}_{\e}$ is a purely spatial partial differential operator of order $r-1$, $\sigmt r A^j_{\e}$ are simply blockdiagonal matrices consisting of $A^j_{\e}$-blocks and $\widetilde B^r_{\e}$ depends solely on $B_{\e}$ and spatial derivatives $\partial_{x_i}A^j_{\e}$. Thus, plugging $\nabla^r U_{\e}$ into  (\ref{special_L2_estimate})  yields
\begin{equation}\label{higher_order_esti_L2_result}
\sup\limits_{0\leq t\leq T}\norm{\nab rU_{\e}(t,\cdot)}{L^2(\mathbb R^n)}^2
 \leq  e^{\beta^{r}_{\e} }\big( \norm{\nab r G_{\e}}{L^2(\mathbb R^n)}^2+
\norm{Q^{r-1}_{\e}\sigm{r}U_{\e}+\nabla^{r} F_{\e}}{L^2(\Omega_T)}^2 \big),
\end{equation}
where $\beta^r_{\e}=1+\norm{\mm{div}\,\sigmt r A_{\e}-\widetilde B^r_{\e}-(\widetilde B^r_{\e})^{\ast}}{L^{1,\infty}(\Omega_T)}$. Iteratively applying (\ref{higher_order_esti_L2_result}) shows that for all $r\in\mathbb N$ there exists $m\in\mathbb N_0$ such that $\norm{\nabla^r U_{\e}}{L^2(\Omega_T)}^2=O(\e^{-m})$. Finally, successively employing equation (\ref{higherorderestitime}) for $l=1,2,3,...$ one finds that the asymptotic growth of $\norm{\partial_t \nabla^r U_{\e}}{L^2(\Omega_T)}^2$, $\norm{\partial_t^2 \nabla^r U_{\e}}{L^2(\Omega_T)}^2$, $\norm{\partial_t^3 \nabla^r U_{\e}}{L^2(\Omega_T)}^2$, ... is also moderate. Altogether we therefore have 
$$
\forall \alpha\in\mathbb N_0^{n+1}
\:\:\exists m\in\mathbb N_0:\:\:\norm{\partial^{\alpha}U_{\e}}{L^2(\Omega_T)}=O(\e^{-m})\quad\mm{as}\quad\e\rightarrow 0$$ 
and hence $[(U_{\e})_{\e}]\in \mathcal G_{L^2}(\Omega_T)$. To show uniqueness, we assume negligible data $(\overline F_{\e})_{\e}\in \mathcal N_{L^2}(\Omega_T)$ and $(\overline G_{\e})_{\e}\in \mathcal N_{L^2}(\mathbb R^n)$. The energy estimates used to prove moderateness then immediately imply negligibility of the corresponding solution $[(\overline U_{\e})_{\e}]$.
\end{proof}

Note that in Theorem \ref{third_theorem}, thanks to the $L^{1,\infty}$-norms in condition ii'') no logarithmic scaling of the mollifier is required to model a lower order coefficient of the form $B(t,x)\approx \delta(t)\widetilde B(x)$ where $\delta$ represents the delta distribution $\widetilde B$ is  bounded.
 
\begin{remark}
For the existence and uniqueness results presented in this section, there  exist versions which are also  global in time. In correspondence with Theorem \ref{main_theorem}, given coefficients $A^j$ and $B$ in $M_m(\mathcal G_{L^{\infty}}(\mathbb R^{n+1}))$, data $F\in\mathcal G(\mathbb R^{n+1})^m$ and $G\in\mathcal G(\mathbb R^n)^m$, one obtains a global solution 
$U\in \mathcal G(\mathbb R^{n+1})^m$ if all $\partial_{x_i}A^j$ as well as the Hermitian part of $B$ are locally  log-type and $A^j_{\e}=O(1)$ as $\e\rightarrow 0$ outside a cylinder $\mathbb R\times B_{R_A}$. Analogously extending the respective asymptotic growth conditions from 
$\Omega_T$ to $\mathbb R^{n+1}$ yields ``global in time''-variants of Theorems \ref{second_theorem} and \ref{third_theorem}. Concerning ``global in time''-solutions,  the requirements on the coefficients with regard to the dependence on $t$ can be somewhat relaxed. It suffices to demand the estimates in the assumptions of the theorems only locally in time, i.e. for all $I\subset\subset \mathbb R$. A similar observation was made  in Remark 1.5.3 in \cite{GKOS:01}. 
\end{remark}

\section{Regularity of the generalized solutions and distributional limits}

To justify the term ``generalized solution'', compatibility with classical smooth and distributional solutions should be investigated. When the coefficients are  smooth, the unique Colombeau solution should be equal to the  respective classical solution in a certain sense. As in \cite{GH:04} and \cite{LO:91}, the following proposition establishes compatibility with the classical results for smooth and distributional data.

\begin{prop}\label{comp_smooth} \label{comp_dist}
\begin{trivlist}
\item{(i)} In Theorem \ref{main_theorem}, additionally assume that $A^j$ and $B$ have components in $ C^{\infty}_b(\overline\Omega_T)$. If $F\in C^{\infty}(\overline\Omega_T)^m$ and $G\in  C^{\infty}(\mathbb R^n)^m$ then the generalized solution $U\in\mathcal G(\Omega_T)^m$ is equal 
to the classical smooth solution.

\item{(ii)} Suppose that $A^j$ and $B$ in Theorem \ref{third_theorem} are smooth. For $s\in \mathbb R$ let $F_0\in L^2([0,T],H^s(\mathbb R^n))^m$ and $G_0\in H^s(\mathbb R^n)^m$ and denote by $U_0$  the unique distributional solution to (\ref{ivp})-(\ref{ivp_ic}) in $ C([0,T],H^s(\mathbb R^n))^m$. Define generalized data by $ F:=[(F_{\e})_{\e}]\in\mathcal G_{L^2}(\Omega_T)^m$ and $ G:=[(G_{\e})_{\e}]\in \mathcal G_{L^2}(\mathbb R^n)^m$, where $F_{\e}$ and $G_{\e}$ are moderate regularizations such that $F_{\e}\rightarrow F_0$ in $L^2([0,T],H^s(\mathbb R^n))^m$ and $G_{\e}\rightarrow G_0$ in $H^s(\mathbb R^n)^m$ as $\e\rightarrow 0$. If $ U=[(U_{\e})_{\e}]$ is the corresponding generalized solution in $\mathcal G_{L^2}(\Omega_T)^m$, then $U_{\e}\rightarrow U_0$ in $ C([0,T],H^s(\mathbb R^n))^m$.
\end{trivlist}
\end{prop}

\begin{proof}
(i) Since we may choose the constant nets $(F)_{\e}$ and $(G)_{\e}$ as representatives of the classes of $F$ and $G$ in $\mathcal G$, we obtain the classical smooth solution to problem (\ref{ivp}-\ref{ivp_ic}) as a representative of the unique Colombeau solution.

(ii) By Theorem 2.6 in \cite{BG:07} we have for all $t\in[0,T]$ 
\begin{multline*}
\norm{U_{\e}(t,\cdot)-U_0(t,\cdot)}{H^s(\mathbb R^n)}^2 \\
 \leq C\big(\norm{G_{\e}-G_0}{H^s(\mathbb R^n)}^2+\int_0^t\norm{F_{\e}(s,\cdot)-F_0(s,\cdot)}{H^s(\mathbb R^n)}^2ds \big),
\end{multline*} 
where the constant $C$ does not depend on $\e$, since all coefficient matrices $A^j$ and $B$ are assumed have components in $C^{\infty}(\overline\Omega_T)$. Taking the supremum over all $t\in[0,T]$ and letting $\e\rightarrow 0$ gives the convergence in $ C([0,T],H^s(\mathbb R^n))^m$. 
\end{proof}

We may interpret Proposition \ref{comp_dist}(ii) as a statement on the regularity of the generalized solution, measured in terms of $H^s$-norms of the associated distribution. Yet this concept of regularity is restricted to situations where distributional limits exist and therefore not applicable if initial data or right-hand side are not associated to any distribution.  
Intrinsic regularity theory in Colombeau algebras is based on the  subalgebra $\mathcal G^{\infty}(\Omega)$  of regular generalized functions in $\mathcal G(\Omega)$ and has been investigated in the context of hyperbolic partial differential equations in  \cite{HOP:05,HK:01,O:08,GO:11,GO:11b}. In the study of intrinsic regularity of generalized solutions to partial differential equations, the notion of \emph{slow scale nets} was introduced in \cite{HO:04} and has proven to be essential in many circumstances (cf. \cite{HOP:05,GO:11,GO:11b}).  
A net $(r_{\e})_{\e}$ of complex numbers is said to be of \emph{slow scale} if $|r_{\e}|^p=O(1/\e)$ as $\e\rightarrow 0$ for all $p\geq 0$.  
As in \cite{GH:04} we call a net $(s_{\e})_{\e}$ of complex numbers a \emph{slow-scale log-type net} if there is a slow scale net $(r_{\e})_{\e}$ of real numbers, $r_{\e}\geq 1$, such that
$ |s_{\e}|=O(\log(r_{\e}))$ as $\e\rightarrow 0$. Generalized functions satisfying the moderateness estimates with slow scale  nets in place of the inverse powers $\e^{-N}$ are called \emph{slow scale regular}. They represent a different notion of regularity than regular generalized functions. 

\begin{prop}\label{prop_internal}
In Theorem \ref{main_theorem}, assume  all coefficients $A^j$ and $B$ to be slow scale regular generalized functions. In addition suppose that all log-type conditions are replaced by slow-scale log-type estimates. Then $F\in\mathcal G^{\infty}([0,T]\times \mathbb R^n)^m$ and $G\in \mathcal G^{\infty}(\mathbb R^n)^m$ implies $U\in \mathcal G^{\infty}(\Omega_T)$.
\end{prop}

\begin{proof}
Fix a standard lens $\mathcal L$ of thickness $T$ with initial surface $\mathcal H_0$. Then there exists $M\in\mathbb N_0$ such that for all $r$, $l$, we have
$$ \norm{\partial_t^l\nabla^r F_{\e}}{L^2(\mathcal L)}=O(\e^{-M}) \quad\textrm{as well as}\quad \norm{\nabla^r G_{\e}}{L^2(\mathcal H_0)}=O(\e^{-M}).
$$ 
Since all coefficient depending factors in the $L^2$-estimates (\ref{L2_estimate_espilon}) and (\ref{higherorder_estimate}) are of slow scale, for each $\alpha\in\mathbb N_0^n$ we obtain a certain  slow-scale net $(r_{\e})_{\e}$ of positive real numbers such that $\norm{\partial_x^{\alpha}U_{\e}}{L^2(\mL)}=O(r_{\e}\e^{-M})$. Finally, by (\ref{higherorderestitime}) we only pick up slow-scale factors with each time derivative applied to $\partial_x^{\alpha}U_{\e}$, and so we have 
$$ \norm{\partial_t^l\partial_x^{\alpha}U_{\e}}{L^2(\mL)}=O(\e^{-M-1})
$$
for all $l\in\mathbb N_0$ and $\alpha\in\mathbb N_0^n$. This  proves the assertion.
\end{proof}

The intrinsic regularity property holds also for the generalized solutions obtained from Theorems \ref{second_theorem} and \ref{third_theorem} respectively, if all log-type conditions are replaced by slow-scale log-type estimates and all coefficients are slow scale regular generalized functions. Note that all these conditions are automatically satisfied when the coefficients are smooth (like in Proposition \ref{comp_smooth} (i)).

Despite its $\mathcal G^{\infty}$-regularity, the solution in Proposition 
\ref{prop_internal} may not be associated to any distribution. In the remaining part of this section we want to investigate distributional limits of  the generalized solution when the coefficients are non-smooth. To this end we consider the generalized Cauchy problem  (\ref{givp_eq}-\ref{givp_in}) with data and coefficient matrices as regularizations.  For convenience of the reader we first collect some basic properties of   regularizations of  $W^{k,p}$-functions obtained by convolution with a mollifier. 

\begin{lemma}\label{lp_reg_via_con}
Let $\rho\in\mathcal D(\mathbb R^n)$, $\int\rho(x)dx=1$, $\mm{supp}(\rho)\subseteq B_1(0)$, and put $\rho_{\e}(x)=\e^{-n}\rho(x/\e)$. Given $u\in W^{k,p}(\mathbb R^n)$, $ k\in\mathbb N_0$, $1\leq p \leq \infty$ we put $u_{\e}:=u\ast \rho_{\e}$. Then we have 
\begin{itemize}
\item[i)]
$[(u_{\e})_{\e}]\in \mathcal G_{L^{\infty}}(\mathbb R^n)$ and $\norm{\partial^{\alpha}u_{\e}}{W^{k,p}(\mathbb R^n)}=O(\e^{-|\alpha|})$,
\item[ii)] $\norm{u_{\e}-u}{W^{k,q}(\mathbb R^n)}\rightarrow 0$ as $\e\rightarrow 0$ for $u\in W^{k,q}(\mathbb R^n)$ with $1\leq q<\infty$, and
\item[iii)] if $u$ is bounded and uniformly continuous, then $\norm{u_{\e}-u}{L^{\infty}(\mathbb R^n)}\rightarrow 0$ as $\e\rightarrow 0$. In particular,  for $u\in W^{k,\infty}(\mathbb R^n)$ we have $\norm{u_{\e}-u}{W^{k-1,\infty}(\mathbb R^n)}\rightarrow 0$ as $\e\rightarrow 0$. 
\end{itemize}
\end{lemma}

\begin{proof}  Here, all norms are taken on the domain $\mathbb{R}^n$.
(i) If $u\in W^{k,p}(\mathbb R^n)$ and $\alpha\in\mathbb N_0^n$ we may estimate  
$\norm{\partial^{\alpha}u_{\e}}{L^{\infty}}=\norm{(\partial^{\alpha}\rho_{\e})\ast u}{L^{\infty}}\leq \norm{\partial^{\alpha}\rho_{\e}}{L^{p'}}\norm{u}{L^p}$ with $1/p+1/p'=1$ by Young's inequality and we have $\partial^{\alpha}\rho_{\e}=\e^{-|\alpha|}(\partial^{\alpha}\rho)_{\e}$ which yields $\norm{\partial^{\alpha}u_{\e}}{L^{\infty}}=O(\e^{-|\alpha|})$, hence $[(u_{\e})_{\e}]\in\mathcal G_{L^{\infty}}(\mathbb R^n)$. Moreover we find 
$$
\norm{\partial^{\alpha}u_{\e}}{W^{k,p}}=\max\limits_{|\beta|\leq k}\norm{\partial^{\alpha+\beta}u_{\e}}{L^p}\leq
\max\limits_{|\beta|\leq k} \norm{(\partial^{\alpha}\rho_{\e})\ast \partial^{\beta}u}{L^p}\leq 
\frac{\norm{\partial^{\alpha}\rho}{L^1}\norm{ u}{W^{k,p}}}{\e^{|\alpha|}}.
$$
 For  (ii) and (iii) we refer to \cite[Theorem 8.14]{Folland:99}.
\end{proof}

The following proposition is concerned with the generalized Cauchy problem  (\ref{givp_eq})-(\ref{givp_in}) when both data and coefficient matrices are obtained from convolution regularizations of non-smooth coefficient matrices $A^j_0$ and $B_0$, right-hand side $F_0$ and initial data  $G_0$.

\begin{prop}\label{prop_dist_1}
In the initial value problem (\ref{givp_eq}-\ref{givp_in}), assume that the entries of $A^j_0$ and  $B_0$ are bounded and Lipschitz continuous with respect to $x$, $G_0\in H^1(\mathbb R^n)^m$, and $F_0\in L^2([0,T],H^1(\mathbb R^{n}))^m$. Define corresponding generalized coefficients and data satisfying the assumptions in Theorem \ref{third_theorem}.
\footnote{The conditions in Theorem \ref{third_theorem} are automatically satisfied when regularizing via convolution with a mollifier as in Lemma \ref{lp_reg_via_con}.}   
Then there exists $U_0\in  C ([0,T],L^2(\mathbb R^n))^m\cap H^1(\Omega_T)^m$ such that $U_{\e}\rightarrow U_0$ in the norm $L^{\infty}([0,T],L^2(\mathbb R^n))$ for every representative $(U_{\e})_{\e}$ of the  unique Colombeau solution $U\in \mathcal G_{L^2}(\Omega_T)^m$. 
\end{prop}

\begin{proof}
First observe that $\norm{\partial_{j}U_{\e}}{L^2(\Omega_T)}=O(1)$ as $\e\rightarrow 0$ for all $j$. For the spatial derivatives $\partial_{x_j}U_{\e}$  this follows from the estimates (\ref{special_esti_beta}) and  (\ref{higher_order_esti_L2_result}) 
by virtue of  the uniform boundedness of $\beta^1_{\e}=1+\norm{\mm{div}\,\sigmt 1 A_{\e}-\widetilde B^1_{\e}-(\widetilde B^1_{\e})^{\ast}}{L^{1,\infty}(\Omega_T)}$  as well as   $(G_{\e})_{\e}$ and $(F_{\e})_{\e}$ in the respective norms (see Lemma \ref{lp_reg_via_con}). The $L^2$-norm of the time derivative $\partial_{t}U_{\e}$   can then be estimated directly by means of the equation. We show that the $\e$-wise constructed family of solutions to the Cauchy problem (\ref{ivp_epsi}-\ref{ivp_epsi_in}), $(U_{\e})_{\e}$, is a Cauchy net in $ C([0,T],L^2(\mathbb R^n))^m$. 
 For indices $0<\tilde \e<\e$ small enough we obtain from (\ref{special_L2_estimate}) 
\begin{eqnarray*}
\sup\limits_{0\leq t\leq T}\norm{U_{\e}(t,\cdot)-U_{\tilde\e}(t,\cdot)}{L^2(\mathbb R^n)}
\leq C_T\big(\norm{G_{\e}-G_{\tilde\e}}{L^2(\mathbb R^n)}+\norm{F_{\e}-P_{\e}U_{\tilde\e}}{L^2(\Omega_T)} \big),
\end{eqnarray*}
where $C_T$ is independent of $\e$ and $P_{\e}=\partial_t+\sum_{j=1}^n A^j_{\e}\partial_{x_j}+B_{\e} =
\partial_t +\sum_{j=1}^n\big(A^j_{\e}-A^j_{\tilde\e}+A^j_{\tilde\e} \big)\partial_{x_j} + B_{\e}-B_{\tilde\e}+B_{\tilde\e}. 
$
Thus, the energy inequality yields 
\begin{multline*}
\norm{U_{\e}-U_{\tilde\e}}{L^{\infty}([0,T],L^2(\mathbb R^n))} 
\leq
C_T\Big(\norm{G_{\e}-G_{\tilde\e}}{L^2(\mathbb R^n)} 
+\norm{F_{\e}-F_{\tilde\e}}{L^2(\Omega_T)}\\ +\sum_{j=1}^n 
\norm{A^j_{\e}-A^j_{\tilde\e}}{L^{\infty}(\Omega_T)}       \norm{\partial_{x_j}U_{\tilde\e}}{L^2(\Omega_T)} 
+\norm{B_{\e}-B_{\tilde\e}}{L^\infty(\Omega_T)}\norm{U_{\tilde \e}}{L^2(\Omega_T)} \Big)
\end{multline*}
where we omitted the domains in the norms on the right-hand side (all norms are taken over the strip $\Omega_T$ or $\mathbb R^n$ respectively). The convergence properties of the nets  $(G_{\e})_{\e}$, $(F_{\e})_{\e}$, $(A^j_{\e})_{\e}$,  $(B_{\e})_{\e}$ and the uniform boundedness of  $\norm{\partial_{x_j}U_{\tilde\e}}{L^2(\Omega_T)}$ imply that $(U_{\e})_{\e}$ is a Cauchy net in the norm $L^{\infty}([0,T],L^2(\mathbb R^n))$, thus there exists $U_0\in  C([0,T],L^2(\mathbb R^n))^m$ such that $\norm{U_{\e}- U_0}{L^{\infty}([0,T],L^2(\mathbb R^n))}\rightarrow 0$ as $\e \rightarrow 0$. Since $\norm{\partial_{j}U_{\e}}{L^2(\Omega_T)}\leq C$, we have 
 for any $\varphi\in\mathcal D(\Omega_T)^m$,
$$ |\langle \partial_{j} U_0,\varphi\rangle |=|\lim_{\e\rightarrow 0}\langle \partial_{j}U_{\e},\varphi\rangle|\leq \lim_{\e\rightarrow 0}|\langle \partial_{j}U_{\e},\varphi\rangle|\leq C \norm{\varphi}{L^2(\mathbb R^n)}.$$
 Hence $\partial_{j}U_0\in L^2(\mathbb R^n)^m$,  thus  $U_0\in  C([0,T], L^2(\mathbb R^n))^m\cap H^1(\Omega_T)^m$.  
\end{proof}

The conditions on the coefficient matrices in Proposition \ref{prop_dist_1} are typical assumptions in results on weak solutions. For a comparison of  solution concepts for linear first-order hyperbolic differential equations with non-smooth coefficients we refer to \cite{HH:08}.

\begin{cor}\label{class_sol_via_col} The initial value problem (\ref{ivp}-\ref{ivp_ic}) with  $A^j_0$, $B_0$, $G_0$ and $F_0$ as in Proposition \ref{prop_dist_1} has a unique weak  solution $U_0\in  C([0,T], L^2(\mathbb R^n))^m\cap H^1(\Omega_T)^m$, equal to the distributional limit of the generalized solution.
\end{cor}

\begin{proof} 
We show that the distributional limit of the generalized solution
$U=[(U_{\e})_{\e}]$ is the unique weak solution. First observe that
the limit $U_0$ of $U_{\e}$ is  continuous in time,  so it satisfies the initial condition, i.e. $U_0|_{t=0}=
G_0$. Moreover we have $U_{\e}\rightarrow U_0$ in $L^2(\Omega_T)$  and thus also $B_{\e}U_{\e}\rightarrow B_0U_0$ in $L^2(\Omega_T)^m$. 
The proof of Proposition \ref{prop_dist_1} shows that $\norm{U_{\e}}{H^1(\Omega_T)}$ is uniformly bounded. By the weak compactness theorem there exists a weakly* convergent subsequence  
$(U_{\frac{1}{n_k}})_{k\in\mathbb N}$ with limit $\widetilde U_0\in H^1(\Omega_T)^m$ (cf. Theorem 6.64 in \cite{RR:03}). Thus $\partial_tU_{\frac{1}{n_k}}\rightarrow \partial_t\widetilde U_0$ weakly* in 
$L^2(\Omega_T)^m$ 
and  $\widetilde U_0 = U_0 \in  C([0,T],L^2(\mathbb R^n))^m\cap H^1(\Omega_T)^m$.  

We have $|\langle \partial_{x_j}U_{\frac{1}{n_k}},(A^j_{\frac{1}{n_k}}-A^j_0)\psi\rangle|\leq \norm{\partial_{x_j}U_{\frac{1}{n_k}}}{L^2(\Omega_T)}\norm{(A^j_{\frac{1}{n_k}}-A^j_0)\psi}{L^2(\Omega_T)}$ for any $\psi\in L^2(\Omega_T)^m$, hence 
$ \langle A^j_{\frac{1}{n_k}}\partial_{x_j}U_{\frac{1}{n_k}},\psi\rangle
= \langle \partial_{x_j}U_{\frac{1}{n_k}},(A^j_{\frac{1}{n_k}}-A^j_0)\psi\rangle+ \langle \partial_{x_j}U_{\frac{1}{n_k}},A^j_0\psi\rangle$ converges to $\langle \partial_{x_j}U_0,A^j_0\psi\rangle$. Therefore we obtain 
\begin{multline*}
\langle (\partial_t +\sum_{j=1}^k A^j_0\partial_{x_j} +B_0) U_0,\psi\rangle\\
= \lim\limits_{\e\rightarrow 0}\Big( \langle \partial_tU_{\e},\psi\rangle + \sum_{j=1}^n \langle A^j_{\e} \partial_{x_j}U_{\e},\psi\rangle +\langle B_{\e}U_{\e},\psi\rangle  \Big)=\lim\limits_{\e\rightarrow 0} \langle F_{\e},\psi\rangle=\langle F_0,\psi\rangle.
\end{multline*}

To prove uniqueness, suppose  $V_0\in   C([0,T],L^2(\mathbb R^n))^m\cap H^1(\Omega_T)^m$ is a solution such that $PV_0=F_0$ and $V_0|_{t=0}=G_0$. We may regularize this solution so that $V_{\e}\rightarrow V_{0}$ in $ C([0,T],L^2(\mathbb R^n))^m$ and in $H^1(\Omega_T)^m$. Then $V_{\e}|_{t=0}\rightarrow G_{0}$ in $L^2(\mathbb R^n)^m$ and necessarily $\lim\limits_{\e\rightarrow 0}\Big( \partial_tV_{\e} + \sum_{j=1}^n  A^j_{\e} \partial_{x_j}V_{\e} +\ B_{\e}V_{\e}\Big)=F_0$ in $L^2(\Omega_T)^m$ irrespective of the regularizations chosen.
Applying the basic energy estimate (\ref{special_L2_estimate}) to the difference $U_{\e}-V_{\e}$ yields
$$ \sup\limits_{0\leq t\leq T}\norm{U_{\e}(t,\cdot)-V_{\e}(t,\cdot)}{L^2(\mathbb R^n)}\leq C_T \big(\norm{G_{\e}-V_{\e}}{L^2(\mathbb R^n)}+\norm{F_{\e}-PV_{\e}}{L^2(\Omega_T)} \big)\rightarrow 0 
$$
as $\e\rightarrow 0$ and therefore $U_0=V_0$ in $ C([0,T],L^2(\mathbb R^n))^m$.
\end{proof}

Note that the statement in Corollary \ref{class_sol_via_col} is not  just a compatibility result, in fact we directly obtain a unique weak solution to the initial value problem with non-smooth coefficients merely by studying properties of the generalized solution when the coefficients and data are regularized distributions. In case of smooth coefficients, the methods applied in the proofs of Proposition \ref{prop_dist_1} and Corollary \ref{class_sol_via_col} would lead to the corresponding classical existence and uniqueness results as well (see Proposition \ref{comp_smooth}).

\begin{remark}
(a) In order to get more information on the regularity of the distributional shadow $U_0$ and its dependence on the regularity of the coefficients $A^j_0$, $B_0$ and data $F_0$, $G_0$, it is essential to have precise estimates on the speed of convergence of the regularized objects (cf. the notion of ``strong association'' in \cite{Scarpi:96}). For example, rapid convergence in the principal part  guarantees the existence of a distributional shadow under regularity assumptions on the initial data $G_0$ and right-hand side $F_0$ which are weaker than those in Proposition \ref{prop_dist_1}. To be more precise, we have the following statement:
In Theorem \ref{third_theorem}, let all $A^j_0$ be Lipschitz continuous and let $B_0$ be uniformly continuous and bounded. Given $G_0\in L^2(\mathbb R^n)^m$ and $F_0\in L^2(\Omega_T)^m$, define generalized coefficients $A$, $B$ and data $F$, $G$  such that
as $\e\rightarrow 0$ 
$$
\norm{A^j_{\e}-A^j_0}{L^{\infty}(\Omega_T)}\mm{max}\big(\norm{\partial_{x_j}B_{\e}}{L^{\infty}(\Omega_T)},\norm{G_{\e}}{H^1(\mathbb R^n)},\norm{F_{\e}}{L^2([0,T],H^1(\mathbb R^n))}\big)\rightarrow 0.
$$
 Then the corresponding generalized solution $U=[(U_{\e})_{\e}]$ still has a distributional shadow $U_0\in L^2(\Omega_T)^m$ such that $U_{\e}\rightarrow U_0$ in $L^2(\Omega_T)^m$ for any of its representatives $(U_{\e})_{\e}$.  

The proof strategy is the same as in Proposition \ref{prop_dist_1}, showing that $(U_{\e})_{\e}$ is a Cauchy net in $L^2(\Omega_T)$, the main difference now being that $\partial_{x_j} U_\e$ is not bounded, but kept under control by the factors stemming from the coefficients.

(b)
To illustrate that the conditions discussed in (a) are realistic, we gather the following estimates on the convergence speed of convolution regularizations: Let $\rho\in\mathcal S(\mathbb R^n)$ with $\int \rho(x)dx=1$ and $\int x^{\alpha}\rho(x)dx=0$ for all $|\alpha|< m\in \mathbb N$, $s\in\mathbb R$ and put $\rho_{\e}:=\sigma_{\e}^{-n}\rho(\frac{\cdot}{\sigma_{\e}})$
where $\sigma_{\e}\rightarrow 0$ as $\e\rightarrow 0$. Then we have
\begin{trivlist}
\item i)  $u\in W^{m,\infty}(\mathbb R^n)$ $\Longrightarrow$  $\norm{\rho_{\e}\ast u-u}{L^{\infty}(\mathbb R^n)}=O(\sigma_{\e}^{m})$ as $\e\rightarrow 0$,
\item ii) $u\in L^1(\mathbb R^n)$ $\Longrightarrow$ $\norm{\rho_{\e}\ast u-u}{H^s(\mathbb R^n)}=O(\sigma_{\e}^{m})$ as $\e\rightarrow 0$ for all $s\leq-\frac{n}{2}-1-m$,
\item iii) $u\in H^{s}(\mathbb R^n)$ $\Longrightarrow$ $\norm{\rho_{\e}\ast u-u}{H^s(\mathbb R^n)}\rightarrow 0$ and $\norm{\rho_{\e}\ast u-u}{H^{s-m}(\mathbb R^n)}=O(\sigma_{\e}^{m})$ as $\e\rightarrow 0$.
\end{trivlist}

Here, i) is shown by Taylor expansion of $u$, ii) can be proved by considering the action on a test function and Taylor expansion of the latter, and iii) follows from the definition of the Sobolev norms in terms of the Fourier transform.     
\end{remark}

\bibliography{chbib}{}

\bibliographystyle{abbrv}

\end{document}

%% file: cs_makros.tex
\newcommand{\mL}{\mathcal L}
\newcommand{\mLT}{\mathcal L_{\Theta}}
\newcommand{\mHT}{\mathcal H_{\Theta}}

\newcommand{\Th}{\Theta}
\newcommand{\mm}{\mathrm}
\newcommand{\nn}{\nonumber}
\newcommand{\e}{\varepsilon}

\newtheorem{theorem}{Theorem}[section]
\newtheorem{nntheorem}{Theorem}
\newtheorem{claim}[nntheorem]{Claim}
\newtheorem{lemma}[theorem]{Lemma}
\newtheorem{prop}[theorem]{Proposition}
\newtheorem{defi}[theorem]{Definition}
\newtheorem{cor}[theorem]{Corollary}
\newtheorem{remark}[theorem]{Remark}

\newcommand{\norm}[2]{{\| #1 \|}_{#2}}
\newcommand{\Norm}[1]{\norm{#1}{}}

\newcommand{\nab}[1]{\nabla^{#1}}
\newcommand{\nabt}[1]{\widetilde\nabla^{#1}\!}
\newcommand{\sigm}[1]{\Sigma^{#1}}
\newcommand{\sigmt}[1]{\widetilde\Sigma^{#1}\!}
